\documentclass[11pt]{article}
\usepackage[margin=1in]{geometry} 
\usepackage{amssymb,amsfonts,amsmath,bbm,mathrsfs,stmaryrd}
\usepackage{xcolor}
\usepackage{url}
\usepackage{braket}

\usepackage[colorlinks,
             linkcolor=blue,
             citecolor=black!75!red,
             pdfproducer={pdfLaTeX},
             pdfpagemode=None,
             bookmarksopen=true
             bookmarksnumbered=true]{hyperref}

\newcommand\bC{\mathbb C}

\newcommand\bG{\mathbb G}

\newcommand\bR{\mathbb R}

\newcommand\bY{\mathbb Y}

\newcommand\cB{\mathcal B}

\newcommand\cH{\mathcal H}

\newcommand\cK{\mathcal K}

\newcommand\cM{\mathcal M}

\usepackage{tikz}
\usetikzlibrary{arrows,calc,decorations.pathreplacing,decorations.markings,intersections,shapes.geometric,through,fit,shapes.symbols,positioning,decorations.pathmorphing}

\usepackage[amsmath,thmmarks,hyperref]{ntheorem}
\usepackage{cleveref}

\crefname{section}{Section}{Sections}
\crefformat{section}{#2Section~#1#3} 
\Crefformat{section}{#2Section~#1#3} 

\crefname{subsection}{\S}{\S\S}
\crefformat{subsection}{#2\S#1#3} 
\Crefformat{subsection}{#2\S#1#3}

\theoremstyle{plain}

\newtheorem{lemma}{Lemma}[section]
\newtheorem{theorem}[lemma]{Theorem}
\newtheorem{proposition}[lemma]{Proposition}
\newtheorem{corollary}[lemma]{Corollary}

\makeatletter
\let\xx@thm\@thm
\AtBeginDocument{\let\@thm\xx@thm}
\makeatother

\theoremstyle{nonumberplain}

\newtheorem{qf'bis}{\Cref{le.quot_filt'} bis}

\theoremstyle{plain}
\theorembodyfont{\upshape}
\theoremsymbol{\ensuremath{\blacklozenge}}

\newtheorem{definition}[lemma]{Definition}

\newtheorem{remark}[lemma]{Remark}

\crefname{definition}{definition}{definitions}
\crefformat{definition}{#2definition~#1#3} 
\Crefformat{definition}{#2Definition~#1#3} 

\crefname{ex}{example}{examples}
\crefformat{example}{#2example~#1#3} 
\Crefformat{example}{#2Example~#1#3} 

\crefname{remark}{remark}{remarks}
\crefformat{remark}{#2remark~#1#3} 
\Crefformat{remark}{#2Remark~#1#3} 

\crefname{convention}{convention}{conventions}
\crefformat{convention}{#2convention~#1#3} 
\Crefformat{convention}{#2Convention~#1#3}

\crefname{lemma}{lemma}{lemmas}
\crefformat{lemma}{#2lemma~#1#3} 
\Crefformat{lemma}{#2Lemma~#1#3} 

\crefname{proposition}{proposition}{propositions}
\crefformat{proposition}{#2proposition~#1#3} 
\Crefformat{proposition}{#2Proposition~#1#3} 

\crefname{corollary}{corollary}{corollaries}
\crefformat{corollary}{#2corollary~#1#3} 
\Crefformat{corollary}{#2Corollary~#1#3} 

\crefname{theorem}{theorem}{theorems}
\crefformat{theorem}{#2theorem~#1#3} 
\Crefformat{theorem}{#2Theorem~#1#3} 

\crefname{assumption}{assumption}{Assumptions}
\crefformat{assumption}{#2assumption~#1#3} 
\Crefformat{assumption}{#2Assumption~#1#3} 

\crefname{equation}{}{}
\crefformat{equation}{(#2#1#3)} 
\Crefformat{equation}{(#2#1#3)}

\theoremstyle{nonumberplain}
\theoremsymbol{\ensuremath{\blacksquare}}

\newtheorem{proof}{Proof}
\newcommand\pf[1]{\newtheorem{#1}{Proof of \Cref{#1}}}

\newcommand{\qedhere}{\mbox{}\hfill\ensuremath{\blacksquare}}

\newcommand\coh{\mathrm{Coh}}

\title{Generic quantum metric rigidity}
\author{Alexandru Chirvasitu}

\begin{document}

\date{}

\newcommand{\Addresses}{{
  \bigskip
  \footnotesize
  
  \textsc{Department of Mathematics, University at Buffalo, Buffalo,
    NY 14260-2900, USA}\par\nopagebreak \textit{E-mail address}:
  \texttt{achirvas@buffalo.edu}
}}

\maketitle

\begin{abstract}
  We introduce the coherent algebra of a compact metric measure space by analogy with the corresponding concept for a finite graph. As an application we show that upon topologizing the collection of isomorphism classes of compact metric measure spaces appropriately, the subset consisting of those with trivial compact quantum automorphism group is of second Baire category.

The latter result can be paraphrased as saying that ``most'' compact metric measure spaces have no (quantum) symmetries; in particular, they also have trivial ordinary (i.e. classical) automorphism group. 
\end{abstract}


\vspace{.5cm}


\section*{Introduction}
\label{se.intro}

Coherent algebras were defined in \cite{hig-ca} as self-adjoint subalgebras of $M_n(\bC)$ that are also closed and unital under the so-called Hadamard matrix product: entry-wise multiplication, i.e. multiplication of matrices regarded as functions $[n]^{\times 2}\to \bC$ for $[n]=\{1,\cdots,n\}$. They arise naturally as algebraic counterparts of certain combinatorial structures (see e.g. \cite[Definition 1.1]{cam-coh}):

\begin{definition}\label{def.coh-conf}
  Let $X$ be a finite set. A {\it coherent configuration} on $X$ is a partition of $X\times X$ into sets $R_i$, $i\in I$ (referred to as the {\it coherence classes} of the configuration) such that
  \begin{itemize}
  \item the diagonal $\Delta\subset X\times X$ is a union of $R_j$;    
  \item for each $R_i$ the opposite relation
    \begin{equation*}
      \overline{R_i}:= \{(x,y)\ |\ (y,x)\in R_i\}
    \end{equation*}
    is one of the $R_j$;
  \item for $i,j,k\in I$ there are integers $p_{ij}^k$ such that for all $(x,z)\in R_k$ the number of
    \begin{equation*}
      y\in X,\quad (x,y)\in R_i\text{ and }(y,z)\in R_j
    \end{equation*}
    is $p_{ij}^k$. 
  \end{itemize}  
\end{definition}

Given a coherent configuration, the span of the $X\times X$-indexed matrices
\begin{equation*}
  A(R_i)_{xy}:=
  \begin{cases}
    1 &\text{if }(x,y)\in R_i\\
    0 &\text{otherwise}
  \end{cases}
\end{equation*}
attached to the relations $R_i$ can then be shown to be a coherent algebra in the sense made precise above. Conversely, every coherent subalgebra of $M_n(\bC)$ arises as above for some coherent configuration on $X=[n]$. In conclusion, coherent algebras and configurations are two manifestations of the same structure.

The preeminent example of a coherent configuration is the partitioning of $X\times X$ into $G$-orbits for the diagonal action attached to an action of a finite group $G$ on $X$ (this is referred to as the {\it group case} in \cite{hig-cc1,hig-cc2,hig-ca}). In that context, coherent algebras have a very satisfying representation-theoretic interpretation (see for instance the discussion on \cite[p. 213]{hig-ca}):

\begin{proposition}\label{pr.init}
  Let $G$ be a finite group acting on the finite set $X$, and denote by $C(X)$ the function algebra of $X$, identified with $\bC^{|X|}$ via the standard basis consisting of characteristic functions of singletons.Then, the isomorphism
  \begin{equation*}
    M_{|X|}(\bC)\cong \mathrm{End}(C(X))
  \end{equation*}
  identifies the coherent algebra $\subseteq M_{|X|}(\bC)$ on the left hand side to the algebra $\mathrm{End}_G(C(X))$ of algebra automorphisms of $C(X)$ intertwining the $G$-action.
  \qedhere
\end{proposition}

It can be shown that the intersection of two coherent subalgebras of $M_n(\bC)$ is again a coherent algebra. If $X$ is the vertex set of a graph there is thus a smallest coherent algebra containing its adjacency matrix; it is customary to refer to it as the {\it coherent algebra of the graph} (see \cite[$\S$2.3]{lmr}).

By \Cref{pr.init}, the automorphism group of a finite graph $X$ is trivial if the coherent algebra of the graph is full (i.e. all of $M_{|X|}(\bC)$). The same principle is applied in \cite{lmr} in the context of actions of compact {\it quantum} groups on finite graphs. We recall the necessary background (e.g. \cite{Wor98,Bic03}) in \Cref{se.prel} below, providing only a broad outline here.

The quantum automorphism group of a graph is a typically non-commutative Hopf algebra $H$ coacting on the function algebra $C(X)$ of the vertex set. The coaction is required to be appropriately compatible with the graph structure (encoded into the adjacency matrix), and the Hopf algebra is regarded as an algebra of functions on the otherwise non-existent ``quantum group'' $G$.   

With this in place, \cite[Theorem 3.11]{lmr} is the quantum analogue of \Cref{pr.init}, and the observation made above, to the effect that full coherent algebras entail trivial automorphism groups, is applied in \cite[Theorem 3.14]{lmr} to argue that as $n$ increases and graph structures are selected uniformly and randomly on an $n$-element set $X$, the probability that the resulting graph has trivial quantum automorphism group approaches $1$. In short:

\begin{equation}
  \label{eq:princ}
  \text{``Most'' finite graphs have trivial quantum automorphism group.}
\end{equation}

The aim of the present paper is to make sense of and prove a continuous analogue of the previous sentence. The phrase `continuous' refers to substituting compact spaces $X$ for finite ones. Furthermore, thinking of a graph structure as a type of low-resolution distance function, we promote our spaces $X$ to compact metric spaces.

Compact quantum groups can act isometrically on compact metric spaces $X$ (see \Cref{se.prel} below). Every such quantum group will also preserve some probability measure on $X$. Since for infinite $X$ there is no canonical choice for such a measure (comparable to the uniform measure on a {\it finite} set $X$), we will also fix a probability measure on $X$ beforehand and focus on measure-preserving actions. All in all, the objects taking the place of finite graphs here are {\it compact metric measure spaces} $(X,d,\mu)$ (`measure' is understood as `probability measure' unless this convention is explicitly discarded). 

It is unclear how one would equip the collection $\bY$ of isomorphism classes of compact metric measure spaces with a probability measure that would then allow us to refer to ``most'' of them, so we instead equip it with a well-behaved topology. In that setup we say that most compact metric measure spaces have a property if the subset of $\bY$ where the property holds is {\it residual}: it contains some dense countable intersections of open sets. 

The rigidity result alluded to in the title can now be stated (see \Cref{th.triv-q}).

\begin{theorem}\label{th.main-init}
Let $\bY$ be the set of isomorphism classes of compact metric measure spaces equipped with the measured Hausdorff topology. Then, the subset consisting of objects with trivial compact quantum automorphism group is residual. 
\end{theorem}

The paper is organized as follows.

\Cref{se.prel} collects some auxiliary material of use later. on both compact quantum groups and metric measure spaces. 

In \Cref{se.innttw} we prove the metric measure space analogue of \Cref{pr.init}. This is \Cref{cor.endg}, obtained as a consequence of \Cref{th.main}. 

\Cref{se.mms} contains the main results of the paper: \Cref{th.triv-coh} shows that the set of isomorphism classes of metric measure spaces with full coherent algebra is large in the sense of Baire category theory, while its consequence, \Cref{th.triv-q}, is the rigorous formulation of the ill-defined principle in \Cref{eq:princ}.

We also argue in \Cref{pr.lpl} that the Laplacian of a finite metric space belongs to its coherent algebra; this recovers the result of \cite[]{}, stating that a compact quantum group acting isometrically on a finite metric space preserves the underlying Laplacian (albeit for an inessentially different incarnation of the Laplacian).

\section{Preliminaries}
\label{se.prel}

\subsection{Compact quantum groups}
\label{subse.cqg}

For background the reader can consult, for instance, \cite{Wor98,KusTus99}, as well as, say, \cite{Tak02} for material on operator algebras. All of our operator algebras are unital, and tensor products between $C^*$ or von Neumann algebras are minimal unless specified otherwise. 

As sketched in the introduction, compact quantum groups $\bG$ are objects dual to certain $C^*$-algebras $C(\bG)$. We formalize this as follows.

\begin{definition}\label{def.cqg}
  A {\it compact quantum group} $\bG$ consists of a unital $C^*$-algebra $C(\bG)$ equipped with a $C^*$ morphism $\Delta:C(\bG)\to C(\bG)^{\otimes 2}$ satisfying the following properties:
  \begin{itemize}
  \item $\Delta$ is {\it coassociative}, in the sense that
    \begin{equation*}
      \begin{tikzpicture}[baseline=(current  bounding  box.center),anchor=base,cross line/.style={preaction={draw=white,-,line width=6pt}}]
    \path (0,0) node (1) {$C(\bG)$} +(3,.5) node (2) {$C(\bG)^{\otimes 2}$} +(6,0) node (3) {$C(\bG)^{\otimes 3}$} +(3,-.5) node (4) {$C(\bG)^{\otimes 2}$}; 
    \draw[->] (1) to[bend left=6] node[pos=.5,auto]{$\scriptstyle \Delta$} (2)  ;
    \draw[->] (2) to[bend left=6] node[pos=.5,auto] {$\scriptstyle \Delta\otimes\mathrm{id}$} (3);
    \draw[->] (1) to[bend right=6] node[pos=.5,auto,swap] {$\scriptstyle \Delta$} (4);
    \draw[->] (4) to[bend right=6] node[pos=.5,auto,swap] {$\scriptstyle \mathrm{id}\otimes\Delta$} (3);    
  \end{tikzpicture}
\end{equation*}
    commutes, and
  \item $\Delta(A)(A\otimes 1)$ and $\Delta(A)(1\otimes A)$ are dense in $A\otimes A$.
  \end{itemize}  
\end{definition}
$C(\bG)$ always has a unique dense Hopf $*$-subalgebra in the sense of \cite[Definition 2.1]{KusTus99}: see \cite[Theorem 3.1.7]{KusTus99}. 

We also need the notion of a unitary representation of a compact quantum group on a Hilbert space (\cite[$\S$3]{Wor98} or \cite[$\S$3.2]{KusTus99}).  

\begin{definition}\label{def.urep}
  A {\it unitary representation} of a compact quantum group $\bG$ on a Hilbert space $\cH$ is a unitary element $U\in \cM(\cK(L^2)\otimes C(\bG))$ satisfying $(\mathrm{id}\otimes \Delta)U=U_{12}U_{13}$,  where
\begin{itemize}
\item $\cK(-)$ denotes compact operators;
\item $\cM(-)$ denotes the multiplier algebra of a non-unital $C^*$-algebra;
\item the ``leg-numbering'' notation $U_{ij}$ indicates the operator $U$ acting on the $i^{th}$ and $j^{th}$ tensorands.  
\end{itemize}
\end{definition}

\subsection{Actions on spaces}
\label{subse.act}

Let $(X,d)$ be a compact metric space, and $\bG$ a compact quantum group acting isometrically on it in the sense of \cite[Definition 3.1]{gos-mtrc}. We will follow the notation of \cite{gos-mtrc}. The coaction consists of a map
\begin{equation*}
  \beta: C(X)\to C(X)\otimes C(\bG). 
\end{equation*}
We will also sometimes consider the purely algebraic counterpart of the picture, whereby the unique dense Hopf $*$-subalgebra $H$ of $C(\bG)$ is substituted for the latter, and the coaction restricts to
\begin{equation*}
  \beta:A\to A\otimes H
\end{equation*}
for a dense $*$-subalgebra $A\subset C(X)$.  

As explained in the discussion on \cite[p. 344]{gos-mtrc}, there is a faithful probability measure $\mu$ on $X$ invariant under $\beta$ in the sense that for every function $f\in C(X)$ we have
\begin{equation*}
  (\mu\otimes \mathrm{id})\beta(f) = \mu(f):=\int_X f\ \mathrm{d}\mu. 
\end{equation*}
We will use a particular measure with these properties throughout the note, and we shorten $L^2(X,\mu)$ to $L^2(X)$ or even $L^2$; similarly, $L^2(X\times X)$ stands for $L^2(X\times X,\mu\boxtimes\mu)$. 

The space $L^2(X)$ is a representation of the compact quantum group, described by a unitary operator $U\in \cM(\cK(L^2)\otimes C(\bG))$, as in \Cref{def.urep}.

As explained on \cite[p. 344]{gos-mtrc}, the condition that the action be isometric is equivalent to either of the equalities
\begin{equation}\label{eq:l}
  U_{13}U_{23}(d\otimes 1) = d\otimes 1 = U_{23}U_{13}(d\otimes 1)
\end{equation}
where $d\otimes 1$ is regarded, say, as an element of $L^2(X\times X)\otimes L^2(\bG)$ (the last tensorand being the GNS space of the Haar measure on $\bG$).

Equivalently, this condition can be phrased as
\begin{equation}\label{eq:lr}
 W(\pi^{(2)}(d)\otimes 1)W^* = \pi^{(2)}(d)\otimes 1 = Z(\pi^{(2)}(d)\otimes 1)Z^*,
\end{equation}
where $W=U_{13}U_{23}$, $Z=U_{23}U_{13}$ and
\begin{equation*}
  \pi^{(2)}=\pi\otimes\pi : C(X\times X)\to \cB(L^2(X\times X))
\end{equation*}
is the tensor-squared GNS representation of $\mu$. 

Note that the set of all continuous functions satisfying \Cref{eq:lr} is a subalgebra of $C(X\times X)$, so is the function algebra of a quotient space $X\times X\to O$. We use this observation as a vehicle for

\begin{definition}\label{def.orb}
  The {\it orbital algebra} associated (or attached) to the action $\beta$ is the subalgebra $C(O)$ of $C(X\times X)$ consisting of functions $f\in C(X\times X)$ which satisfy \Cref{eq:lr} upon substituting $f$ for $d$. 
  
  The {\it orbitals} of the action $\beta$ are the points of the spectrum $X\times X\to O$ of the orbital algebra. 
\end{definition}

\begin{remark}
As noted above, it follows from the discussion on \cite[p. 344]{gos-mtrc} that for any function $f\in C(X\times X)$ the four equalities in \Cref{eq:l,eq:lr} are mutually equivalent. This, together with \cite[Lemma 3.8]{lmr}, also shows that for finite metric spaces $X$ the present notion of orbital coincides with that of \cite[Definition 3.5]{lmr}. 
\end{remark}

\subsection{Metric measure spaces}
\label{subse.mm}

Throughout this note the phrase `metric measure space' refers to a compact metric space $(X,d)$ equipped with a faithful probability measure $\mu$. For such a space we consider two multiplicative structures on the space $C(X\times X)$ of continuous functions: the standard multiplication of functions denoted by juxtaposition, and the convolution by $\mu$ denoted by $*$ and used above:
\begin{equation*}
  (f*g)(x,z) = \int_X f(x,y)g(y,z)\ \mathrm{d}\mu(z). 
\end{equation*}
We denote the $n^{th}$ power of a function $f$ under convolution by $f^{*n}$.

We will now topologize the collection $\bY$ of (the isomorphism classes of) all metric measure spaces adapting the notion of {\it measured Hausdorff topology} from \cite[Definition 0.2]{fuk} to the present setting so as to take into account both structures (metric and measure-theoretic). In order to do this, we need some preparation.

Recall (e.g. \cite[Definition 7.3.27]{bbi})

\begin{definition}\label{def.eiso}
  For $\varepsilon>0$ and {\it $\varepsilon$-isometry} $f:X\to Y$ between metric spaces $(X,d_X)$ and $(Y,d_Y)$ is a possibly discontinuous function such that
  \begin{itemize}
  \item the $\varepsilon$-neighborhood of $f(X)$ is all of $Y$, and
  \item $|d_Y(f(x),f(x'))-d_X(x,x')|\le \varepsilon$ for all $x,x'\in X$.    
  \end{itemize}
\end{definition}
We can define a quasimetric on $\bY$ (i.e. distance function satisfying all of the usual axioms except perhaps for symmetry) by
\begin{equation}\label{eq:dd}
  D_d((X,d_X),(Y,d_Y)) = \inf\{\varepsilon>0\ |\ \text{ there exists an }\varepsilon-\text{isometry }X\to Y\}. 
\end{equation}
$D_d$ is equivalent to the Gromov-Hausdorff distance (e.g. by \cite[Corollary 7.3.28]{bbi}) and hence induces the same topology on $\bY$. 

We will use a finer topology on $\bY$, that also carries measure-theoretic information. Given metric measure spaces $(X,d_X,\mu_X)$ and $(Y,d_Y,\mu_Y)$ and an $\varepsilon$-isometry $f:X\to Y$ we will need to compare the probability measures $f_*\mu_X$ and $\mu_Y$ on $Y$. We will use the weak$^*$ topology on $\mathrm{Prob}(Y)$ for this purpose, in the sense that a net $\mu_\alpha$ converges to $\mu$ if 
\begin{equation*}
  \mu_\alpha(\varphi)\to \mu(\varphi)
\end{equation*}
for all continuous functions $f\in C(Y)$. This topology is metrizable in numerous ways, e.g. via the so-called {\it Wasserstein distances} \cite[Definition 6.1]{Vil09} $W_p$ defined (on $X$, for $p\ge 1$) by
\begin{equation*}
  W_p(\mu,\nu) = \inf_\pi \left(\int_{X\times X}d(x,x')^p\ \mathrm{d}\pi(x,x') \right)^{\frac 1p},
\end{equation*}
where $\pi$ ranges over the {\it $(\mu,\nu)$-couplings}: those probability measures on $X\times X$ whose two marginals are $\mu$ and $\nu$ respectively. Any $p\ge 1$ will do for our purposes, but we retain flexibility by incorporating $p$ into the definition of he following quasimetric.  

\begin{definition}\label{def.Dp}
  Let $(X,d_X,\mu_X)$ and $(Y,d_Y,\mu_Y)$ be elements of $\bY$, and $p\ge 1$. We then write $D_p(X,Y)$ for the infimum of the set of those $t>0$ for which there exists a measurable $t$-isometry $f:X\to Y$ such that
  \begin{equation*}
    W_p(f_*\mu_X,\mu_Y)\le t.
  \end{equation*}  
\end{definition}

We topologize $\bY$ by defining the neighborhoods of $X\in \bY$ to be those sets that contain some
\begin{equation*}
  U_{\varepsilon,X}=\{Y\in \bY\ |\ D_p(X,Y)<\varepsilon\}
\end{equation*}
for some $\varepsilon>0$. The choice of $p\ge 1$ makes no difference here. $\bY$ is a Baire space in the sense that it satisfies Baire's theorem: countable intersections of open dense sets are again dense. We record this here for future use. 

\begin{lemma}\label{le.baire}
  $\bY$ topologized via any of the quasimetrics $D_p$ is a Baire space.  
\end{lemma}
\begin{proof}
  The proof of \cite[Corollary 7.3.28]{bbi} to the effect that the quasimetric $D_d$ of \Cref{eq:dd} induces the Gromov-Hausdorff topology and hence is equivalent to its transpose
  \begin{equation*}
    D_d^t(X,Y) = D_d(Y,X)
  \end{equation*}
  extends to the present setting to show that the topological space $\bY$ can be metrized by $D+D^t$. By Baire's theorem, it thus suffices to argue that $D+D^t$ is complete.

  In turn, this last claim follows from the fact that convergence with respect to $D+D^t$ dominates the Gromov-Hausdorff distance (e.g. \cite[Corollary 7.3.28]{bbi}), together with the fact that the Gromov-Hausdorff distance and $W_p$ are both complete.
\end{proof}

\begin{remark}\label{re.evans-winter}
  The same topology, restricted to compact {\it $\bR$-trees}, is considered in \cite{ew-metric} and metrized in essentially the same fashion. \cite[Theorem 2.5]{ew-metric} similarly proves that the resulting metric space is complete and separable.
\end{remark}

Due to \Cref{le.baire}, any subset of $\bY$ containing a dense $G_\delta$ is appropriately thought of as ``large''. We will use the term {\it residual} for such sets. See e.g. \cite{oxt} for some background on this and ancillary concepts,  and parallels to measure-theoretic notions of largeness; chapter 9 therein contains the definition of residual sets, in agreement with the present one.

\section{Intertwiners}\label{se.innttw}

In the above discussion we regarded $C(X\times X)$ as an algebra of operators on $L^2(X\times X)$ via the GNS representation. We now consider its representation
\begin{equation*}
  \rho: C(X\times X)\to \cB(L^2(X))
\end{equation*}
by convolution through $\mu$:
\begin{equation*}
  (\rho(f)(g))(x) = \int_X f(x,y)g(y)\ \mathrm{d}\mu(y)
\end{equation*}
for $f(x,y)\in C(X\times X)$ and $g(y)\in L^2(X)$. This is analogous to the action of the $X\times X$-indexed matrix algebra on the space of functions on $X$ for finite $X$. 

Recall also that an {\it intertwiner} for the representation $U$ of $\bG$ on $L^2(X)$ is an element $t$ of $\cB(L^2(X))$ satisfying $Ut=tU$ if we regard $t$ as an element of $\cM(\cK(L^2)\otimes C(\bG))$ by
\begin{equation*}
\cB(L^2(X))\cong \cM(\cK(L^2(X)))\to \cM(\cK(L^2)\otimes C(\bG)).  
\end{equation*}
The intertwiners of $U$ form a von Neumann subalgebra $\mathrm{End}_G(L^2(X))$ of $\cB(L^2(X))$. Since furthermore every representation of $\bG$ on a Hilbert space decomposes as a direct sum of finite-dimensional irreducible representations, we have
\begin{equation*}
  L^2(X) \cong \bigoplus_i H_i\otimes V_i
\end{equation*}
as $\bG$-representations, where $V_i$ are irreducible (and hence finite-dimensional) and $H_i$ are multiplicity spaces, i.e. carry the trivial $\bG$-representation. This in turn entails
\begin{equation*}
  \mathrm{End}_G(L^2(X))\cong \prod_i \cB(H_i),
\end{equation*}
where the product is in the category of von Neumann algebras. The embedding of this product into $\cB(L^2(X))$ via the standard inclusion
\begin{equation*}
  \mathrm{End}_G(L^2(X))\subset \cB(L^2(X))
\end{equation*}
is such that the minimal projections of $\cB(H_i)$ have respective rank $\mathrm{dim}(V_i)$. We will use this observation below. 

The goal of this section is to prove the following partial analogue of \cite[Theorem 3.11]{lmr}.

\begin{theorem}\label{th.main}
  For a function $f\in C(X\times X)$, $\rho(f)$ is an intertwiner for $U$ if and only if $f$ belongs to the orbital algebra $C(O)$. 
\end{theorem}
\begin{proof}
  We prove the two implications separately.

  \vspace{.5cm}

  {\bf ($\Leftarrow$)} Suppose first that $f$ belongs to the orbital algebra, i.e. the equivalent conditions \Cref{eq:l,eq:lr} hold with $f$ in place of $d$. We'll want to verify that
  \begin{equation}\label{eq:conj}
    \rho(f) = U^{-1}\rho(f) U. 
  \end{equation}
Operating with the right hand side of \Cref{eq:conj} on the function $g\in C(X)$ is equivalent to applying $\mu$ to the middle leg of 
\begin{equation}\label{eq:fbg}
  U^{-1}_{13}(f\otimes 1)(1\otimes \beta(g)).
\end{equation}
The left hand equality in \Cref{eq:l} ensures that $U^{-1}_{13}(f\otimes 1)=U_{23}(f\otimes 1)$. If we write $f_x(y)=f(x,y)$, this is an $x$-dependent element of $C(X)\otimes C(\bG)$ given by
\begin{equation*}
  x\mapsto \beta(f_x). 
\end{equation*}
In conclusion, the right hand side of \Cref{eq:fbg} is the $x$-dependent element of $C(X)\otimes C(\bG)$
\begin{equation*}
  x\mapsto \beta(f_x)\beta(g) = \beta(f_xg). 
\end{equation*}
Applying $\mu$ to the left hand tensorand of the right hand side in this equation produces
\begin{equation*}
  x\mapsto (\mu\otimes 1)\beta(f_xg) = x\mapsto \mu(f_xg)
\end{equation*}
by the $\beta$-invariance of $\mu$. This, however, is nothing but the result of operating with $\rho(f)$ on $g$ (by the very definition of $\rho$).

  \vspace{.5cm}

  {\bf ($\Rightarrow$)} Conversely, suppose
  \begin{equation*}
    U\rho(f) = \rho(f)U. 
  \end{equation*}
  We can now reverse the argument above. The hypothesis is that for every $g\in C(X)$, the results of operating on $g$ with $\rho(f)$ and $U^{-1}\rho(f)U$ coincide. As before, application of the latter amounts to evaluating $\mu$ against the middle tensorand of \Cref{eq:fbg}. On the other hand, applying the former means evaluating $\mu$ on the middle leg of
  \begin{equation}\label{eq:fbg2}
    U_{23}(f\otimes 1)(1\otimes\beta(g))
  \end{equation}
  instead. Since these results are moreover unaffected by further multiplying the rightmost tensorand on the right by arbitrary $t\in C(\bG)$, the rightmost factors in \Cref{eq:fbg,eq:fbg2} can be replaced by arbitrary $1\otimes(\beta(g)(1\otimes t))$. Since every element of $A\otimes H$ is of the form $\beta(g)(1\otimes t)$, we may as well assume that the right hand factors in \Cref{eq:fbg,eq:fbg2} are arbitrary tensor products $1\otimes g\otimes t$ (or indeed simply $1\otimes g\otimes 1$).

The conclusion now follows from the fact that $\mu$ is faithful, and hence, if $a=U_{13}^{-1}(f\otimes 1)$ and $a'=U_{23}(f\otimes 1)$ denote the left hand factors of \Cref{eq:fbg,eq:fbg2} respectively, the vanishing of all
  \begin{equation*}
    (1\otimes\mu\otimes 1)((a-a')(1\otimes g\otimes 1))
  \end{equation*}
  for arbitrary $g\in C(X)$ entails the vanishing of $a-a'$. In other words, as claimed, $f$ satisfies \Cref{eq:l}.
\end{proof}

\cite[Theorem 3.11]{lmr} in fact identifies the entire algebra $\mathrm{End}_G(L^2(X))$ with the orbital algebra $C(O)$ (see also \cite[$\S$7]{hig-cc1} for the case of ordinary as opposed to quantum permutation groups). On the other hand, \Cref{th.main} only realizes $C(O)$ as a subalgebra of $\mathrm{End}_G(L^2(X))$; for this reason, the analogy thus far is only partial. The inclusion
\begin{equation*}
 C(O)\subseteq \mathrm{End} 
\end{equation*}
is indeed proper when $X$ is infinite, but only for the simple reason that the larger algebra is von Neumann, while the smaller one is only a $C^*$-algebra of continuous functions. The two differ in just this obvious fashion though:

\begin{corollary}\label{cor.endg}
  The von Neumann closure of $C(O)$ in $\cB(L^2(X))$ is $\mathrm{End}_G(L^2(X))$. 
\end{corollary}
\begin{proof}
  The argument in the proof of \Cref{th.main} applies to functions in $L^2(X\times X)$ to argue that they too belong to $\mathrm{End}_G(L^2(X))$ precisely when they are constant along the fibers of $X\times X\to O$.

  On the other hand, recall from the discussion preceding \Cref{th.main} that $\mathrm{End}_G(L^2(X))$ is generated as a von Neumann algebra by finite-rank operators. The latter are contained in the algebra
  \begin{equation*}
    L^2(X\times X)\cong L^2(X)\otimes L^2(X)
  \end{equation*}
  of Hilbert-Schmidt operators on $L^2(X)$, meaning that $\mathrm{End}_G(L^2(X))$ is the von Neumann closure of its subalgebra of Hilbert-Schmidt operators. By the first paragraph of the present proof, this coincides with the $W^*$ closure of $C(O)$.
\end{proof}

\section{Coherent algebras for metric measure spaces}
\label{se.mms}

Following the discussion in \cite[$\S$2.3]{lmr} in the context of finite graphs, we introduce the following object.

\begin{definition}\label{def.coh}
  The {\it coherent algebra} $\coh(X)$ of a metric measure space $(X,d,\mu)$ is the smallest unital $C^*$-subalgebra (with respect to the standard algebra structure) of $C(X\times X)$ meeting the following requirements
  \begin{itemize}
  \item it contains $d$;    
  \item it is closed under the convolution multiplication $*$;    
  \item it is closed under flips: if $f\in \coh(X)$ then so is the function $(x,y)\mapsto f(y,x)$. 
  \end{itemize}
\end{definition}

\begin{remark}
  For a finite graph, its coherent algebra as recalled in \cite[$\S$2.3]{lmr} coincides with $\coh(X)$ in the sense of \Cref{def.coh} when $X$ is equipped with its path metric and the uniform probability measure (i.e. the rescaled counting measure). 
\end{remark}

Now consider the setup of the previous section, of a compact quantum group acting isometrically (and faithfully) on $(X,d,\mu)\in \bY$. In the language of \Cref{th.main} $\rho(d)\in \cB(L^2(X))$ is an intertwiner for $U$, and hence the coherent algebra $\coh(A)$ is a subalgebra of $C(O)$. In conclusion:

\begin{lemma}\label{le.triv-orb}
  If the coherent algebra of $(X,d,\mu)$ equals all of $C(X\times X)$ then $(X,d,\mu)$ has trivial quantum automorphism group. 
\end{lemma}
\begin{proof}
  Indeed, in that case $C(O)=C(X\times X)$, and hence the latter is contained in the algebra of intertwiners for the $\bG$-representation $L^2(X)$. Since the weak-$*$ closure of $C(X\times X)$ in $\cB(L^2(X))$ contains all Hilbert-Schmidt operators, the intertwiner algebra must be all of $\cB(L^2(X))$. 

It now follows that the (faithful) representation $U$ of $\bG$ on $L^2(X)$ is trivial, which in turn means that $\bG$ is trivial. 
\end{proof}

\begin{theorem}\label{th.triv-coh}
  The subset $\bY_{triv}$ of $\bY$ consisting of metric measure spaces with full coherent algebra is residual. 
\end{theorem}
\begin{proof}
  We will argue that $\bY_{triv}$ contains a dense $G_\delta$ set. Since we are working in a complete metric space, this amounts to showing that it contains a countable intersection of open dense subsets of $(\bY,d_T)$.

  Fix positive integers $N$, $m$ and $p$ and denote by $\bY_{N,m,p}$ the subset of $\bY$ consisting of (classes of) metric spaces $(X,d,\mu)$ for which the following condition holds:
  \begin{equation}\label{eq:nmp}
    \text{Whenever }d(x,x')+d(y,y')\ge \frac 1m\text{ we have }|d^{*n}(x,y)-d^{*n}(x',y')|>\frac 1p\text{ for some }1\le n\le N. 
  \end{equation}
  We then have
  \begin{equation*}
    \bY_{triv} = \bigcap_{m}\bigcup_{N,p} \bY_{N,m,p},
  \end{equation*}
  and hence it suffices to argue that each $\bigcup_{N,p}\bY_{N,m,p}$ is open and dense in $(\bY,d_T)$. We prove these two claims separately.

  \vspace{.5cm}

  {\bf $\bigcup_{N,p}\bY_{N,m,p}$ is dense.} First, note that the set of isomorphism classes of finite metric measure spaces is dense in $\bY$. Furthermore, the metric of any finite metric space can be perturbed by arbitrarily small amounts as to ensure that the distance function $d:X\times X\to [0,\infty)$ is one-to-one off the diagonal $\Delta$ of $X$ and $d^{*2}$ is one-to-one on $\Delta$. Such finite metric measure spaces $(X,d,\mu)$ thus constitute a dense subset of $\bY$. On the other hand that set is contained in $\bigcup_{N,p}\bY_{N,m,p}$ for every positive $m$, hence the density claim.  

  \vspace{.5cm}
  
  {\bf $\bY_{N,m,p}$ is open.}

  Suppose not. This means that we can find a metric measure space $(X,d,\mu)$ in $\bY_{N,m,p}$ and a sequence $(X_n,d_n,\mu_n)$ of elements of $\bY\setminus \bY_{N,m,p}$ converging to $X$ in the topology we have equipped $\bY$ with.

  The assumptions mean that we can find points $x_n,x'_n,y_n,y'_n$ in $X_n$ such that
  \begin{enumerate}
    \renewcommand{\labelenumi}{(\arabic{enumi})}
  \item $d_n(x_n,x'_n) + d_n(y_n,y'_n)\ge \frac 1m$;
  \item for all $1\le i\le N$ we have $|d_n^{*i}(x_n,y_n)-d_n^{*i}(x'_n,y'_n)|\le \frac 1p$.    
  \end{enumerate}
  On the other hand, the convergence $(X_n,d_n,\mu_n)\to (X,d,\mu)$ means that there are $\varepsilon_n$-isometries $f_n:X_n\to X$ as in \Cref{def.Dp} for $\varepsilon_n\to 0$.

By compactness, after passage to a subsequence we can assume that $f_n(x_n)$ converge to $x\in X$ and similarly for $x'$, $y$ and $y'$. For each individual $i$ we then have
  \begin{equation*}
    d_n^{*i}(x_n,y_n)\to d^{*i}(x,y)\text{ and } d_n^{*i}(x'_n,y'_n)\to d^{*i}(x',y'),
  \end{equation*}
  and hence conditions (1) and (2) above imply respectively that
  \begin{itemize}
  \item $d(x,x') + d(y,y')\ge \frac 1m$;
  \item for all $1\le i\le N$ we have $|d^{*i}(x,y)-d^{*i}(x',y')|\le \frac 1p$.   
  \end{itemize}
  This, however, contradicts our assumption that $(X,d,\mu)$ belongs to $\bY_{N,m,p}$ and finishes the proof.
\end{proof}

\begin{theorem}\label{th.triv-q}
  The subset of $\bY$ consisting of metric measure spaces with trivial quantum automorphism groups is residual. 
\end{theorem}
\begin{proof}
  Simply combine \Cref{th.triv-coh,le.triv-orb}. 
\end{proof}

\subsection{Laplacians and coherent algebras}
\label{subse.lpl}

The main result of the present subsection is

\begin{proposition}\label{pr.lpl}
Let $(X,d)$ be a finite metric space equipped with the uniform probability measure. Then, the Laplace operator $\Delta$ on $X$ belongs to $\coh(X)$.    
\end{proposition}

In particular, since as observed above we have $\coh(X)$ is contained in the orbital algebra $C(O)$ for any isometric action of a compact quantum group $G$ on $(X,d,\mu)$, isometric quantum actions automatically preserve the Laplacian. This gives an alternate take on \cite[Theorem 4.5]{gsw-0}, stating that the quantum isometry groups of finite metric spaces can be recovered as implementing quantum symmetries of natural Laplace operators attached to said metric spaces. 
Before going into the proof of \Cref{pr.lpl} we recall some background on Laplacians using \cite{rie-res} as our source. Following \cite[$\S$2]{rie-res}, we write
\begin{equation*}
  Z:=\{(x,y)\in X^2\ |\ x\ne y\}
\end{equation*}
for the off-diagonal of the Cartesian square $X^2$. We have a ``gradient'' operator
\begin{equation*}
  \partial:C(X)\to C(Z),\ \partial f(x,y) = f(x)-f(y)
\end{equation*}
and an inner product on $C(Z)$ defined by
\begin{equation*}
\braket{\omega,\omega'}:=\int_X \left(\sum_{x\ne y}\overline{\omega(x,y)}\omega'(x,y)c_{xy}\right)\ d\mu(y)
\end{equation*}
where $c_{xy}=\frac 1{d(x,y)}$ are the ``conductances'' attached to the metric, along with the usual inner product on $C(X)$ itself induced by $\mu$: 
\begin{equation*}
\braket{f,f'}:= \int_X \overline{f(x)}f'(x)\ d\mu(x).   
\end{equation*}

Finally, we have

\begin{definition}\label{def.lpl}
  The {\it Laplacian} on $X$ is the unique positive operator $\Delta$ on $L^2(X)$ satisfying
  \begin{equation*}
    \braket{\partial f,\partial f'} = \braket{f,\Delta f'},\ \forall f,f'\in C(X).
  \end{equation*}
\end{definition}

Since integration against $\mu$ is simply the averaging operator
\begin{equation*}
  \int_X f\ d\mu(x) = \frac 1{|X|} \sum_{x\in X}f(x)
\end{equation*}
the definition of $\Delta$ expands as
\begin{equation*}
  \sum_{z\in X} \overline{f(z)}\Delta f'(z) = \sum_{x\ne y}\overline{(f(x)-f(y))}(f'(x)-f'(y))c_{xy}
\end{equation*}
which in turn equals
\begin{equation*}
  2\sum_{x\ne y}\overline{f(x)}(f'(x)-f'(y))c_{xy}. 
\end{equation*}
All in all:
\begin{equation*}
  \Delta f(x) = 2\sum_{y\ne x}(f(x)-f(y))c_{xy}. 
\end{equation*}
In conclusion, written as an $X\times X$-indexed matrix operating on $L^2(X)\cong \bC^{|X|}$, $\Delta$ is (up to a global scalar we henceforth disregard) equal to $T+A$ where $T$ is diagonal with 
\begin{equation*}
  T_{xx} = \sum_{y\ne x}c_{xy} 
\end{equation*}
and $A$ is off-diagonal (i.e. its main diagonal is zero) and $A_{xy}=-c_{xy}$ for $x\ne y$.

\pf{pr.lpl}
\begin{pr.lpl}
  According to the discussion carried out above it will suffice to show that both $T$ and $A$ belong to $\coh(X)$. We prove these claims separately.

  {\bf (1): $A\in \coh(X)$. } We have $I+A=(I+D)^{-1}$ with respect to the Hadamard product where $I$ is the identity matrix and $D$ is the distance matrix with $D_{xy}=d(x,y)$. Since both $I$ and $D$ are contained in $\coh(A)$ and the latter is a finite-dimensional $C^*$-algebra under the Hadamard product the inverse
  \begin{equation*}
    (I+D)^{-1} = I+A
  \end{equation*}
  is a member of $\coh(A)$ as well, and hence so is $A$.

  {\bf (2): $T\in \coh(X)$.} Write $X\times X\to O$ for the surjection corresponding to the embedding $\coh(X)\subset C(X\times X)$, so that $\coh(X)=C(O)$. We have to argue that viewed as a function on $X\times X$, $T$ factors through $X\times X\to O$. In other words, the task is to show that $T\in C(X\times X)$ is constant along the equivalence classes $R_i$ making up the coherent configuration corresponding to $\coh(X)$ (see \Cref{def.coh-conf}). 

  Recall that the diagonal $\{(x,x)\ |\ x\in X\}$ is a union of classes $R_i$. Since the off-diagonal entries of $T$ vanish it will be enough to argue that $T$ is constant along the diagonal relations $R_i$.

  Suppose $(x,x)$ and $(y,y)$ belong to the same class. \Cref{le.aux} applied to singleton intervals shows that for every $r\in \bR$ the number of points $x\ne z\in X$ with $d(x,z)=r$ equals the analogous number for $y$. In other words\begin{equation*}
\{c_{xz},\ x\ne z\in X\}\text{ and } \{c_{yw},\ y\ne w\in X\}
\end{equation*}
coincide as multisets , and hence the $(x,x)$ and $(y,y)$ entries
\begin{equation*}
  \sum_{z\ne x}c_{zx}\text{ and }\sum_{w\ne y}c_{wy}
\end{equation*}
of $T$ coincide.
\end{pr.lpl}

\begin{lemma}\label{le.aux}
  Let $(X,d)$ be a finite metric space and $(x,x)$, $(y,y)$ diagonal points in the same equivalence class of the canonical coherent configuration attached to $\coh(X)$.

  Then, for any interval $[a,b]\subset\bR$ the sets
  \begin{equation}\label{eq:xysets}
    \{z\in X\ |\ d(z,x)\in [a,b]\}\quad \text{ and }\quad \{w\in X\ |\ d(w,y)\in [a,b]\}
  \end{equation}
  have the same cardinality. 
\end{lemma}
\begin{proof}
  Let $D$ be the distance matrix, as before (i.e. $D_{xy}=d(x,y)$). Since $D\in \coh(X)$, applying Borel functional calculus applied to the $C^*$-subalgebra $\coh(X)\subset C(X\times X)$ yields
  \begin{equation*}
    f(D):=\text{the matrix with entries }f(D_{xy})\in \coh(X). 
  \end{equation*}
This holds in particular for $f=\chi_{[a,b]}$ (the characteristic function of $[a,b]$), leading to $M\in \coh(A)$ where
  \begin{equation*}
    M_{x,y}=
    \begin{cases}
      1,& \text{if }d(x,y)\in [a,b]\\
      0,& \text{otherwise}
    \end{cases}
  \end{equation*}
  
  Now, the square $M^2$ with respect to the usual, matrix multiplication is also a member of $\coh(X)$, and note that its diagonal entry $M^2_{xx}$ is exactly the cardinality of the set on the left hand side of \Cref{eq:xysets} (and similarly for $M^2_{yy}$). Since $(x,x)$ and $(y,y)$ are assumed in the same coherence class, $M^2\in \coh(X)$ must assign them equal value, as claimed.
\end{proof}


\bigskip

\subsection*{Acknowledgments}
This work was partially supported by NSF grant DMS-1801011.

I would like to thank Teodor Banica, Debashish Goswami, Martino Lupini, Laura Man{\v c}inska and David Roberson for valuable input on the contents of the paper and related matters.


\bibliographystyle{plain}

\def\polhk#1{\setbox0=\hbox{#1}{\ooalign{\hidewidth
  \lower1.5ex\hbox{`}\hidewidth\crcr\unhbox0}}}

\Addresses

\end{document}